\documentclass[review]{elsarticle}

\usepackage{lineno,hyperref}
\usepackage{amsfonts,amssymb}
\usepackage{amsmath,amsthm}
\usepackage{graphicx}
\usepackage{enumerate}
\usepackage[algoruled]{algorithm2e}

\newtheorem{theorem}{Theorem}
\newtheorem{lemma}{Lemma}
\newtheorem{corollary}{Corollary}
\newtheorem{proposition}{Proposition}
\newtheorem{definition}{Definition}

\begin{document}

\begin{frontmatter}

\title{$D$-Magic Strongly Regular Graphs}

\author{Rinovia Simanjuntak}
\address{Combinatorial Mathematics Research Group, Institut Teknologi Bandung, Indonesia}
\ead{rino@math.itb.ac.id}

\author{Palton Anuwiksa}
\address{Master Program in Mathematics, Institut Teknologi Bandung, Indonesia}
\ead{anuwiksapalton@gmail.com}

\begin{abstract}
For a set of distances $D$, a graph $G$ on $n$ vertices is said to be $D$-magic if there exists a bijection $f:V\rightarrow \{1,2, \ldots , n\}$ and a constant $k$ such that for any vertex $x$, $\sum_{y\in N_D(x)} f(y) = k$, where $N_D(x)=\{y|d(x,y)=i, i\in D\}$ is the $D$-neighbourhood set of $x$. 	 
In this paper we utilize spectra of graphs to characterize strongly regular graphs which are $D$-magic, for all possible distance sets $D$. In addition, we provide necessary conditions for distance regular graphs of diameter 3 to be $\{1\}$-magic.
\end{abstract}


\end{frontmatter}

\section{Introduction}\label{intro}

The notion of distance magic labeling was introduced in the PhD thesis of Vilfred \cite{Vi} in 1994. A {\em distance magic labeling} of a graph $G$ on $n$ vertices is a bijection $f:V(G) \rightarrow \{1,2, \ldots, n\}$ with the property that there exists a \emph{magic constant} $k$ such that at any vertex $x$, its \emph{weight} $w(x) = \sum_{y\in N(x)} f(y) = k$, where $N(x)$ is the open neighborhood of $x$, i.e., the set of vertices adjacent to $x$. A graph admitting a distance magic labeling is then called  \emph{distance magic}. In 1999, Jinah \cite{Ji99} introduced a variation of distance magic labeling called the \emph{closed distance magic labeling}, where the weight of a vertex $x$ is summed over the closed neighborhood of $x$, i.e., the set containing $x$ and all vertices adjacent to $x$. O'Neal and Slater \cite{OS13} then generalized both aforementioned labelings by introducing the $D$-magic labeling, where $D$ is a set of distances.

\begin{definition}
Let $G$ be a graph on $n$ vertices, with diameter $d$. Let $D \subseteq \{0,1,\ldots,d\}$ is a set of distances in $G$. $G$ is said to be \emph{$D$-magic} if there exists a bijection $f:V\rightarrow \{1,2, \ldots , n\}$ and a \emph{magic constant }$k$ such that for any vertex $x$, $\sum_{y\in N_D(x)} f(y) = k$, where $N_D(x)=\{y|d(x,y)=d, d\in D\}$ is the \emph{$D$-neighbourhood set of $x$}.
\end{definition}

It was proven in \cite{Vi} that every graph is a subgraph of a distance magic graph, which showed that there is no forbidden subgraph characterization for distance magic graph. One of the most important results for $D$-magic labeling is in \cite{OS13,AKV14} where it is shown that for a particular graph, the magic constant is unique and is determined by its fractional domination number. For more results, refer to survey articles in \cite{AFK11} and \cite{Ru}.

A \emph{distance-regular graph} is a regular connected graph with degree $r$ and diameter $d$, for which there are positive integers \[b_0=r, b_1, \ldots, b_{d-1}, c_1=1, c_2, \ldots, c_d,\] such that for any two vertices $u,v$ at distance $i$, there are precisely $c_i$ neighbours of $u$ in $G_{i-1}(v)$ ($1\leq i\leq d$) and $b_i$ neighbours of $u$ in $G_{i+1}(v)$ ($0\leq i\leq d-1$). The array \[\iota(G):=\{r,b_1,\ldots,b_{d-1};1,c_2,\ldots,c_d\}\] is called the \emph{intersection array} of $G$. A distance-regular graph of diameter at most 2 is called a \emph{strongly regular graph}. We say that an $r$-regular graph is \emph{strongly regular with parameters $(r,a,c)$} if each pair of adjacent vertices has $0\leq a \leq r$ common neighbours and each pair of non-adjacent vertices has $c\geq 1$ common neighbours.

It is easy to see that if a graph is $D$-magic then it is also $D^*$-magic, where $D^*=\{0,1,\ldots,d\}-D$. Since the diameter of strongly regular graphs is at most two, then by considering $D\in\{\{1\},\{0,1\}\}$ we already have all possible $D$-magic labelings for strongly regular graphs.

In this paper, we shall characterize $D$-magic strongly regular graphs and strongly regular graphs whose line graphs are $D$-magic (Section \ref{SRG}). We shall also provide necessary conditions such that a distance regular graph of diameter 3 is $D$-magic (Section \ref{diam3}). In order to do so, first we shall consider the spectra of distance magic regular graphs in Section \ref{spectra}. In some of the proofs, we shall use computational tools to show inexistence of distance magic labelings for a particular graph. The algorithm of the labeling search is presented in Section \ref{algo}.

\section{Spectra of Graphs and Distance Magic Labelings} \label{spectra}

In \cite{ACP16}, a necessary condition for a regular graph to be closed distance magic was presented; from which a characterization of closed distance magic strongly regular graphs was obtained. An analog result for distance magicness will be proved here.

\begin{proposition} \cite{ACP16}
If $G$ is a regular graph and has closed distance magic labeling, then $-1$ is an eigen value of $G$.
\label{-1}
\end{proposition}

\begin{theorem} \cite{ACP16} \label{CDMSRG}
A strongly regular graph $G$ on $n$ vertices is closed distance magic if and only if $G \approx K_n$.
\end{theorem}

Let us consider a graph $G$ admitting a distance magic labeling $f$ with magic constant $k$. Then the following matrix equation holds.
\begin{equation} \label{adjacencyeq}
A \mathbf{f} = k \mathbf{1},
\end{equation}
where $A$ is the adjacency matrix of $G$, $\mathbf{f}$ is a vector of all labels under the labeling $f$, and $\mathbf{1}$ is a vector of all ones.

Equation (\ref{adjacencyeq}) is essential in finding the following necessary condition for the existence of distance magic regular graphs.
\begin{theorem}
If $G$ is a regular graph admitting a distance magic labeling then 0 is an eigenvalue of $G$.
\label{0}
\end{theorem}
\begin{proof} Let $G$ be an $r$-regular graph. Thus in addition to the matrix equation (\ref{adjacencyeq}), another matrix equation holds, i.e., $A \mathbf{h} = k \mathbf{1}$, where $\mathbf{h} = (\frac{k}{r} \ \frac{k}{r} \ \ldots \ \frac{k}{r})^t$. Hence, $A(\mathbf{f}-\mathbf{h}) = 0$, which means $0$ is an eigen value of $G$.
\end{proof}

If a graph contains two vertices $x$ and $y$ with the same open neighborhood then the vector $\mathbf{u}$ whose only non-zero elements are $1$ and $-1$ corresponding to the vertices $x$ and $y$ is an eigenvector of $G$ with eigenvalue 0. An obvious example of such a graph is the complete multipartite graphs. However, we can see in the next theorem that there are infinitely many complete multipartite graphs which are not distance magic and this shows that the necessary condition in Theorem \ref{0} is not sufficient.

\begin{theorem} \cite{MRS03}
Let $H_{n,p}$ be the complete multipartite graph with $p$ partite sets, each partite set consists of $n$ vertices.
For $n>1$ and $p>1$, $H_{n,p}$ has distance magic labeling if only if $n$ is even or both $n$ and $p$ are odd.
\label{Hnp}
\end{theorem}

In \cite{Ji99}, Jinnah proved that $G$ has distance magic labeling if and only if $\overline{G}$ has closed distance magic labeling;
and so, combining with Theorem \ref{-1}, we obtain the following.
\begin{corollary}
If $G$ is a regular graph and $\overline{G}$ is distance magic then $-1$ is an eigen value of $G$.
\end{corollary}

A necessary condition for line graphs of regular graphs to admit closed distance magic labelings is presented in \cite{ACP16}. We shall provide an analog condition for distance magic labelings.

\begin{theorem} \cite{ACP16}
Let $G$ be an $r$-regular graph. If the line graph of $G$, $L(G)$, is closed distance magic, then $G$ has $1-r$ as an eigen value.
\end{theorem}

\begin{theorem} \label{LG}
Let $G$ be an $r$-regular graph. If the line graph of $G$, $L(G)$, is distance magic, then $G$ is either a $K_2$ or has $2-r$ as an eigen value.
\end{theorem}
\begin{proof}
Let the eigen values of $G$ be $r, \lambda_1, \ldots, \lambda_{s-1}$, then the eigen values of $L(G)$ are $2r-2, r-2+\lambda_1, \ldots, r-2+\lambda_{s-1}, -2$. Since the line graph of regular graph is also regular, then, if $L(G)$ has distance magic labeling, by Theorem \ref{0}, $0$ is an eigen value of $L(G)$. Therefore, either $2r-2=0$ or there exists an $i$ such that $r-2+\lambda_i=0$. If $2r-2=0$ then $r=1$ and $G \approx K_2$. In the second case, $2-r$ is an eigen value of $G$.
\end{proof}

It is known that the complete graph of order $n$, $K_n$, has $-1$ as an eigenvalue (of multiplicity $n-1$). In particular, for $K_4$ which is a $3$-regular graph, one of its eigenvalues is $-1=2-3$. In this case, $L(K_4)\approx H_{2,3}$ is distance magic by Theorem \ref{Hnp}. On the other hand, for $n\equiv 0 \bmod 4$, the cycle $C_n$, which is $2$-regular, has $0=2-2$ as an eigenvalue. However, $L(C_n)\approx C_n$ is not distance magic for $n\neq 4$ (see \cite{MRS03}), which shows that the necessary conditions in Theorem \ref{LG} are not sufficient.

\section{A Naive Exhaustive Search} \label{algo}

In the next section, there is a need to prove that a regular graph is not distance magic. To do so, we utilise a naive exhaustive search of distance magic labelings of the regular graph. To actually search for a distance magic labeling for a particular graph, a heuristic approach might be more efficient; one such approach can be found in \cite{YS15}.

Two properties of distance magic regular graphs will be used in the algorithm, one connected to its magic constant and the other to its degree. Both properties were proved in \cite{MRS03}. As mentioned in the Introduction, the magic constant of a graph is unique and is determined by its fractional domination number. For a regular graph, the magic constant is easier to count.
\begin{proposition} \cite{MRS03} \label{constant}
If an $r$-regular graph of order $n$ admits a distance magic labeling then the magic constant is $k=\frac{r(n+1)}{2}$.
\end{proposition}

\begin{proposition} \cite{MRS03} \label{odddeg}
If an $r$-regular graph is distance magic then $r$ is even.
\end{proposition}

\begin{algorithm}[H]
\SetAlgoLined
\KwIn{$n$: order of a regular graph $G$; $r$: degree of $G$; $\{x_1,x_2,\ldots,x_n\}$: the set of vertices of $G$; $A$: the adjacency matrix of $G$}
\KwOut{a distance magic labeling $f$ of $G$ (if exists)}
 \eIf{$r$ is odd}{
  \Return{"$G$ is not distance magic"}\;
  }{
  $P \gets$ the set of all permutations of the set $\{1,2,\ldots,n\}$\;
  $k \gets \frac{r(n+1)}{2}$\;
  \While{$P\neq \emptyset$}{
   choose $f\in P$\;
   \For{$i \gets 1$ \textbf{to} $n$}{
   label $x_i$ with $f[i]$\;
   }
   $i \gets 1$\;
   $w \gets w(x_1)$ (the weight of $x_1$)\;
   \While{$w=k$ \textbf{and} $i<n+1$}{
    $i \gets i+1$\;
    $w \gets w(x_i)$ (the weight of $x_i$)\;
    }
   \eIf{$i>n$}{
    \Return{$f$ "is a distance magic labeling of $G$"}\;
    \textbf{stop}\;
    }{
    $P \gets P\setminus \{f\}$\;
    }
   }
  \Return{"$G$ is not distance magic"}\;
  }
 \caption{Distance Magic Labeling Search}
\end{algorithm}

\newpage
The algorithm could easily be altered so it could be used to search for a closed distance magic labeling of a particular graph,  by adjusting the vertex-weight formula and utilising the following properties instead of Propositions \ref{constant} and \ref{odddeg}.
\begin{proposition}
If an $r$-regular graph of order $n$ admits a closed distance magic labeling then the magic constant is $k=\frac{(r+1)(n+1)}{2}$.
\end{proposition}

\begin{proposition}
If an $r$-regular graph of order $n$ is closed distance magic then $n+r$ is odd.
\end{proposition}

\section{$D$-Magic Labelings for Strongly Regular Graphs} \label{SRG}


Applying Theorem \ref{0} to strongly regular graphs leads to the complete characterization of distance magic strongly regular graphs.
\begin{theorem} \label{DMSRG}
Let $G$ be a strongly regular graph of degree $r$, on $n$ vertices. $G$ admits a distance magic labeling if and only if $G$ is a regular complete multipartite graph $H_{n-r,\frac {n}{n-r}}$, where $n-r$ is even or both $n-r$ and $\frac {n}{n-r}$ are odd.
\end{theorem}
\begin{proof}
By Theorem \ref{0}, if G admits a distance magic labeling then $G$ has $0$ as an eigenvalue. It is well known that a strongly regular graph has eigenvalue 0 if and only if it is a complete multipartite graph (see, for instance, \cite{BCN}). Since $G$ is $r$-regular, then the cardinality of each partite set is $n-r$, and so $G \approx H_{n-r,\frac {n}{n-r}}$. We complete the proof by applying Theorem \ref{Hnp}.
\end{proof}

Combining the results in Theorems \ref{DMSRG} and \ref{CDMSRG}, we obtain the following characterization of $D$-magic strongly regular graphs.
\begin{theorem}
A strongly regular graph $G$ is $D$-magic if and only if
\begin{enumerate}
  \item $D=\{0,1\}$ and $G \approx K_n$, or
  \item $D=\{0,2\}$ and $G \simeq H_{n-r,\frac {n}{n-r}}$, where $n-r$ is even or both $n-r$ and $\frac {n}{n-r}$ are odd, or
  \item $D=\{0,1,2\}$, or
  \item $D=\{1\}$ and $G \simeq H_{n-r,\frac {n}{n-r}}$, where $n-r$ is even or both $n-r$ and $\frac {n}{n-r}$ are odd, or
  \item $D=\{2\}$ and $G \approx K_n$.
\end{enumerate}
\end{theorem}

We also obtain the following characterization of strongly regular graphs whose line graphs are distance magic.
\begin{theorem}
Let $G$ be a strongly regular graph. The line graph of $G$ is distance magic if and only if $G$ is a cycle of order 4.
\end{theorem}
\begin{proof} Let $G$ be a strongly regular graph of order $n$ with parameters $(r,a,c)$, where $c\geq 1$ and $0 \leq a < r$.
Let the eigen values of $G$ be $r,x_1,x_2$, where $x_1$ and $x_2$ are the roots of $x^2+(c-a)x+(c-r)=0$.
Consequently, the eigen values of $L(G)$ are $2r-2,r-2+x_1,r-2+x_2,-2$. If $L(G)$ is distance magic, then by Theorem \ref{0},
either $2r-2=0, r-2+x_1=0$, or $r-2+x_2=0$.

We shall consider two cases: $r=2$ and $r>2$.

If $r=2$, without loss of generality, $r-2+x_1=0$, and so $x_1=0$, which means $G$ is a complete multipartite graph. Since $x_1 x_2=c-r$ then $c=r=2$, and so $G\approx C_4$, where its line graph, $L(C_4) \approx C_4$ is distance magic.

If $r>2$, without loss of generality, $r-2+x_1=0 $, and $x_1=2-r$. Since $x_1+x_2=-(c-a)$, then $x_2=a-c-2+r$. From $x_1 x_2=c-r$, we obtain
\begin{equation}
(r-2)(r-c+a-2)=r-c
\label{L(DRG)}
\end{equation}
If $r-c+a-2=0$, then $c=r$, and so $a=2$. Therefore, $G$ has parameters $(r,2,r)$. It is well known that if $G$ is a strongly regular graph of order $n$ with $c=r$ then $G \approx H_{(n-r),\frac{n}{n-r}}$ and $a=2r-n$ (see, for instance, \cite{BCN}). Therefore, $G\approx K_{2,2,2}$. If $r-c+a-2 \neq 0$, from (\ref{L(DRG)}), $r-c\neq 0$. Since $r\geq c$ then $r-c+a-2>0$. Applying again equation (\ref{L(DRG)}), we obtain $r-c \geq r-2$, which means $c\leq 2$. The case of $c=1$ leads to $G$ with parameters $(3,2,1)$. On the other hand, the case of $c=2$, leads to $G$ with parameters $(3,2,2), (4,1,2)$, or $(5,0,2)$. Parameters $(3,2,1)$ and $(3,2,2)$ are not possible, while parameters $(4,1,2)$ and $(5,0,2)$ lead to exactly two graphs: the $(3\times3)$-grid on 9 vertices \cite{BH} and the Clebsch graph on 16 vertices \cite{GR}.

Our exhaustive search of distance magic labelings for the line graphs of the three remaining graphs: the complete multipartite graph $K_{2,2,2}$, the $(3\times3)$-grid, and the Clebsch resulted in the negative.
\end{proof}

Similar to the case of distance magic labelings, we could show that line graphs of strongly regular graphs do not admit closed distance magic labelings, except for the trivial case of the cycle on 3 vertices.
\begin{theorem}
Let $G$ be a strongly regular graph. The line graph of $G$ is closed distance magic if and only if $G$ is a cycle of order 3.
\end{theorem}
\begin{proof}
Let $G$ be a strongly regular graph of order $n$ with parameters $(r,a,c)$, where $c\geq 1$ and $0 \leq a < r$. Let the eigen values of $G$ be $r,x_1,x_2$, where $x_1$ and $x_2$ are the roots of
\begin{equation} \label{SRGeigeneq}
x^2+(c-a)x+(c-r)=0.
\end{equation}
Consequently, the eigen values of $L(G)$ are $2r-2,r-2+x_1,r-2+x_2,-2$. If $L(G)$ is closed distance magic, then by Theorem \ref{-1}, either $2r-2=-1$ (impossible), $r-2+x_1=-1$, or $r-2+x_2=-1$. Without loss of generality, let $r-2+x_1=-1$, and so $x_1=1-r$. Substituting into (\ref{SRGeigeneq}), we obtain
$(1-r)^2+(c-a)(1-r)+(c-r)=0$. Adding $(1-c)$ to both sides results in $(1-r)^2+(c-a)(1-r)+(1-r)=1-c$, and so
\begin{equation}
(r-1)(c+2-r-a)=c-1.
\label{CDM}
\end{equation}

Now we consider two cases: $c=1$ and $c\geq 2$. If $c=1$, then $(r-1)(3-r-a)=0$. Since $r\geq 2$, then $r+a=3$. Thus the possible parameters of $G$ are $(2,1,1)$ or $(3,0,1)$; both of which are not realisable.

If $c\geq 2$, then $c+2-r-a\geq 1$. From (\ref{CDM}), $r-1 \leq c-1$ or $r\leq c$, and we have $r=c$, which means $G$ is a regular complete multipartite graph. Substituting $r=c$ into Equation (\ref{CDM}) results in $a=1$, and the fact that $G$ is a regular complete multipartite graphs leads that $G\approx C_3$ which obviously is closed distance magic.
\end{proof}

The results in the last two theorems provide the following characterization of strongly regular graphs whose line graphs are $D$-magic.
\begin{theorem}
Let $G$ be a strongly regular graph. The line graph of $G$, $L(G)$, is $D$-magic if and only if
\begin{enumerate}
  \item $D=\{0,1\}$ and $G \approx C_3$, or
  \item $D=\{0,2\}$ and $G \approx C_4$, or
  \item $D=\{0,1,2\}$, or
  \item $D=\{1\}$ and $G \approx C_4$, or
  \item $D=\{2\}$ and $G \approx C_3$.
\end{enumerate}
\end{theorem}

\section{Distance Magic Labeling for Distance Regular Graphs of Diameter 3} \label{diam3}

We shall conclude by presenting necessary conditions for the existence of distance magic labelings for distance regular graphs of larger diameter, in particular of diameter 3.

\begin{lemma} \label{det}
Let $G$ be a distance regular graph with diameter $d$ and intersection array $$\{r,b_1,\ldots,b_{d-1};1,c_2,\ldots,c_d\}.$$
If $G$ is distance magic then
\[det \left(
\begin{array}{rrrrr}
a_2 & c_3 &     &    &0\\
b_2 & a_3 & c_4 &    &\\
& b_3 & a_4 &  \ddots  &\\	
&     &     &  \ddots & c_d\\
0     &     &     &    & a_d
\end{array}
\right)=0,\]
where $a_i=r-b_i-c_i, i\in \{1,2,\ldots,d\}$.
\end{lemma}
\begin{proof}
The eigen values of $G$ are also the eigen values of the tridiagonal matrix
\[\texttt{T}= \left(
	\begin{array}{rrrrrrrr}
	0 & 1  &      &     &          &    0\\
	r & a_1 & c_2 &     &          &     \\
	  & b_1 & a_2 & c_3 &          &      \\
 	  &     & b_2 & a_3 & \ddots   &      \\
	  &     &     &     & \ddots   &   c_d   \\
   	0 &     &     &     &          &   a_d   \\
	\end{array}
	\right),\]
where $a_i=r-b_i-c_i, i \in \{1,2,3,\ldots,d\}$.
Since $0$ is an eigen value of $G$ then $det(\texttt{T}-0\cdot I)=0$, where $I$ is the identity matrix. This completes the proof.
\end{proof}

\begin{theorem}
Let $G$ be a distance-regular graph of diameter 3 with intersection array $\{r,b_1,b_2;1,c_2,c_3\}$.
If $G$ is distance magic then $G$ is primitive and $(r-b_2-c_2)(k-c_3)=b_2 c_3$.
\end{theorem}
\begin{proof}
Clearly, the following two equations hold for the parameters of the intersection array:
$$r=b_2+c_2+a_2,$$
$$r=c_3+a_3.$$

On the other hand, since $G$ is distance magic then by Lemma \ref{det},
$$det \left(
\begin{array}{rrrrrrrr}
a_2 & c_3 \\
b_2 & a_3 \\
\end{array}
\right)=0,$$
and so
\begin{equation} \label{b2c3}
(r-b_2-c_2)(k-c_3)=b_2 c_3.
\end{equation}

If $G$ is either antipodal or bipartite, then $c_3=r$ (see \cite{Bi}), which leads to a contradiction to the positivity of both $b_2$ and $c_3$.
\end{proof}

Referring to the BCN tables of feasible intersection arrays for distance-regular graphs on at most 4096 vertices \cite{BCNtable}; from 105 feasible intersection arrays for primitive graphs of diameter 3 and even degree, only 13 arrays fullfiling Equation (\ref{b2c3}). The feasible intersection arrays are listed in Table \ref{13}.

\begin{table}[h] \label{13}
\begin{center}
\begin{tabular}{c|c|c|c}
  \textbf{No} & \textbf{Intersection array} & \textbf{Order} & \textbf{Graph} \\
  \hline
  1 & \{6,4,2;1,2,3\} & 27 & Hamming graph $H(3,3)$ \\
  2 & \{10,6,4;1,2,5\} & 65 & Hall graph from $P\Sigma L(2,25)$ \\
  3 & \{12,6,2;1,4,9\} & 35 & Johnson graph $J(7,3)$ \\
  4 & \{12,10,3;1,3,8\} & 68 & Doro graph from $P\Sigma L(2,16)$ \\
  5 & \{30,22,9,1,3,20\} & 350 & unknown  \\
  6 & \{36,25,16;1,4,18\} & 462 & unknown \\
  7 & \{40,33,8;1,8,30\} & 250 & unknown \\
  8 & \{42,30,12;1,6,28\} & 343 & unknown \\
  9 & \{60,42,18;1,6,40\} & 670 & unknown \\
  10 & \{60,45,8;1,12,50\} & 322 & unknown \\
  11 & \{60,52,10;1,10,48\} & 438 & unknown \\
  12 & \{72,45,16;1,8,54\} & 598 & unknown \\
  13 & \{72,70,8;1,8,63\} & 783 & unknown \\
  \hline
\end{tabular}
\caption{Feasible intersection arrays for primitive graphs of diameter 3 and even degree with $(r-b_2-c_2)(k-c_3)=b_2 c_3$.}
\end{center}
\end{table}

Applying a naive exhaustive search to graphs in Table \ref{13} is somewhat not feasible, and so it would be nice to have a more efficient algorithm; a pruned exhaustive search might be one possible approach.

To conclude, the problem of finding all distance regular graphs which are $D$-magic is still wildly open. The result for a special case of distance regular graphs, the hypercubes, can be found in \cite{AMS}.

\section*{References}


\end{document}